\documentclass[reqno]{amsart}

\usepackage{amssymb}

\newtheorem{theorem}{Theorem}[section]
\newtheorem{lemma}[theorem]{Lemma}
\newtheorem{proposition}[theorem]{Proposition}
\newtheorem{corollary}[theorem]{Corollary}
\theoremstyle{definition}

\theoremstyle{remark}
\newtheorem{remark}[theorem]{Remark}
\numberwithin{equation}{section}

\newcommand{\R}{{\mathbb R}}

\newcommand{\G}{\frac{\alpha}{\Gamma(1-\alpha)}}

\newcommand{\C}{\mathcal{C}}

\author{Mark Allen}
\address{Department of Mathematics, Brigham Young University, Provo,
  UT 84602}
\email{allen@mathematics.byu.edu}

\thanks{The author was supported by NSF grant DMS-1303632. Much of the work was completed while the author was at The University 
of Texas at Austin.}

\title{A nondivergence parabolic problem with a fractional time derivative}

\begin{document}

\maketitle
\makeatletter
\vspace{-2em}
{\centering\enddoc@text}
\let\enddoc@text\empty 
\makeatother
\begin{abstract}
 We study a nonlocal nonlinear parabolic problem with a fractional time derivative. We prove a Krylov-Safonov type result; mainly, we prove
 H\"older regularity of solutions. Our estimates remain uniform as the order of the fractional time derivative $\alpha \to 1$. 
\end{abstract}

\section{Introduction}
 This paper studies a nonlocal nondivergence nonlinear parabolic equation with a fractional time derivative. The specific equation we will study is
  \begin{equation}   \label{e:e1}
   _a^c\partial_t^{\alpha} u - I u = \ _a^c\partial_t^{\alpha} u - \sup_{i}\inf_{j}\left( \int_{\R^n}\frac{u(x+y)-u(x)}{|y|^{n+2\sigma}} a^{ij}(t,x,y) \right) = f
  \end{equation}
 where $i,j$ are two indexes ranging in arbitrary sets, and $_a^c \partial_t^{\alpha}$ represents the Caputo derivative. We assume  
  \begin{gather*}
    a^{ij}(t,x,y) = a^{ij}(t,x,-y) \\
    \lambda \leq a^{ij}(t,x,y) \leq \Lambda.
  \end{gather*}
 The symmetry condition assures that the integral is convergent as long as $u \in C^{1,\alpha}$. A similar parabolic problem involving the Caputo
 derivative was recently studied in \cite{acv16} where the nonlocal spatial operator was linear and of divergence type. The authors obtained a H\"older regularity result 
 using De Giorgi's method. Another fractional parabolic problem involving a local divergence operator was studied in \cite{z13}. 
  
 For a parabolic nonlocal nondivergence setting (with the classical time derivative) 
 a more complicated Hamilton-Jacobi type equation with $\sigma=1/2$ was studied in \cite{s11}. 
 For a local time derivative when $0<\sigma< 1$
 the problem was studied in \cite{cl141} and \cite{cl142}. In those papers the authors were able to prove the results uniformly as $\sigma \to 1$ and hence 
 generalizing the results in the local setting. In this paper our focus is on the fractional Caputo derivative. Our results remain uniform
 as $\alpha \to 1$. Since we are not focused on the results being uniform as $\sigma \to 1$ we follow the methods presented in \cite{s11}. 
 
 \subsection{Motivation}
  Nonlocal equations are effective in modeling phenomena in the physical sciences. Fractional diffusion operators can be used to model problems involving  Levy flights, 
  for instance the fractional kinetic equations in \cite{z02}. 
  Fractional kinetic equations are also derived from the viewpoint of random walks. 
  A fractional diffusion operator is used for a diverging jump length variance in the random walk, and a fractional time derivative is utilized 
  when the characteristic
  waiting time diverges, see \cite{mk00}. 
  A fractional time derivative is advantageous in physical models where the problem has ``memory''. Some nonlocal equations have mixed spatial and time derivatives such as
  the so called master equations in \cite{mk00}. Other models have separate nonlocal spatial and time derivatives such as the model we are studying.
  
  One motivation for the equation we study comes from modeling transport in plasma physics. In \cite{dcl04} and \cite{dcl05} the authors used the Caputo time derivative
  as well as a nonlocal spatial operator to model transport of tracer particles in plasma. The nonlocal spatial operator accounts for avalanche like
  transport that can occur. The Caputo time derivative accounts for the trapping effects of turbulent eddies. 
  
  Of interest is the regularity of solutions to these nonlocal equations. H\"older continuity for solutions to master equations was shown in \cite{cs14}. 
  H\"older continuity
  for equations of type \eqref{e:e1} with the usual local time derivative was shown in \cite{cl141}. H\"older continuity for parapoblic equations with a 
  fractional time derivative and divergence form nonlocal operator was shown in \cite{acv16}. For a nonlocal porous medium equation involving
  the Caputo derivative, H\"older continuity was proven in \cite{2acv16}. Our main result in this paper is to prove H\"older continuity for 
  solutions to \eqref{e:e1} under the appropriate assumptions (see Theorem \ref{t:holder}). It is important to note that the estimates remain uniform as 
  the order of the fractional derivative $\alpha \to 1$. 
  
 \subsection{The Caputo Derivative} 
  As previously mentioned the Caputo derivative is effective in modeling equations in which there is memory. One formulation of the Caputo derivative
  is 
   \[
    _a^c\partial_t^{\alpha} f(t) := \frac{1}{\Gamma(1-\alpha)} \int_{a}^t \frac{f'(s)}{(t-s)^{1+\alpha}} \ ds.
   \]
  For $C^1$ functions one may use integration by parts to show the equivalent formula
   \[
    _a^c\partial_t^{\alpha} f(t) = \frac{1}{\Gamma(1-\alpha)} \frac{f(t)-f(a)}{(t-a)^{\alpha}} + \G \int_a^t \frac{f(t)-f(s)}{(t-s)^{1+\alpha}} \ ds.
   \]
  If we define $f(t)=f(a)$ for $t<a$ as in \cite{acv16}, then we have the equivalent formulation 
   \begin{equation}  \label{e:fracform}
    \partial_t^{\alpha} f(t) := \G \int_{-\infty}^t \frac{f(t)-f(s)}{(t-s)^{1+\alpha}} \ ds.
   \end{equation}
  This one-sided nonlocal derivative was recently studied in \cite{bmst16}. The formulation in \eqref{e:fracform} is very useful. It is no longer essential to know, and therefore label, 
  the initial point $a$. (We will also
  omit the superscript $c$ and always assume that the fractional derivative is the Caputo derivative.) Another feature of this formulation in \eqref{e:fracform}
  is that  one may easily  
  assign more general ``initial'' data than just $u_0(x) = u(x,a)$ that need not be differentiable or Lipschitz. 
  One may assign initial data $g(x,t)$ for any $t<a$. The formulation in \eqref{e:fracform} is also suited for 
  viscosity solutions. This formulation in \eqref{e:fracform} looks similar to the one-dimensional fractional Laplacian 
  except that the integration occurs from only one side. 
  Another feature is the formula used in Lemma \ref{l:byparts} to treat the Caputo derivative of a product of functions. 
  Finally, this formulation in \eqref{e:fracform} allows for a different type of generalization of the Caputo derivative. Rather than generalizing as
   \[
    \frac{d}{dt} (k \ast (f-f(a))) \quad \text{ or } \int_a^t f'(s)K(t-s) \ ds
   \]
  as is commonly done, one may generalize as 
   \[
    \int_{-\infty}^t [f(t)-f(s)]K(t,s) \ ds. 
   \]
  In fact, the proof in \cite{acv16} for the linear divergence equation works for a general fractional time derivative of the form
   \begin{equation}   \label{e:gentime}
    \int_{-\infty}^t [f(t)-f(s)]K(t,s,x),
   \end{equation}
  provided that the kernel satisfies $K(t,t-s,x) = K(t+s,t,x)$ and
   \begin{equation}  \label{e:alphabound}
    \frac{\Lambda^{-1}}{(t-s)^{1+\alpha}} 
     \leq K(t,s,x)
     \leq \frac{\Lambda}{(t-s)^{1+\alpha}}.
   \end{equation}
  The first condition is utilized since the problem is of divergence form.

  What would be of interest is to prove 
  H\"older regularity to solutions of an equation of the form
   \begin{equation}  \label{e:future}
    \begin{aligned}
     &\int_{-\infty}^t [u(x,t)-u(x,s)]K_1(t,s,x) \\
     &\quad  - \sup_{i}\inf_{j}\left( \int_{\R^n}u(x+y,t)-u(x,t)K_2^{ij}(t,x,y) \right) = f(t,x) 
    \end{aligned}
   \end{equation}
  or even 
   \begin{equation}  \label{e:nontime}
    \begin{aligned}
     &\sup_{k}\inf_{l}\left( \int_{-\infty}^t [u(x,t)-u(x,s)]K_1^{kl}(t,s,x)  \right) \\
      & \quad - \sup_{i}\inf_{j}\left( \int_{\R^n}u(x+y,t)-u(x,t)K_2^{ij}(t,x,y) \right) = f(t,x)
    \end{aligned}
   \end{equation}
  assuming only the bounds from \eqref{e:alphabound} on $K_1^{kl}$ and the appropriate ellipticity bounds on $K_2^{ij}$. Of greatest interest would be to prove
  the results with uniform estimates as $\alpha \to 1$ and $\sigma \to 2$. The second equation seems to be a nonlinear nondivergence analogue for the fractional time 
  derivative. In a forthcoming paper \cite{a17} the author has shown H\"older continuity for solutions to \eqref{e:future} provided that $K_1(t,s,x) =K_1(t,s)$, so 
  that the kernel is independent of $x$ and also assuming $K(t,t-s)=K(t+s,t)$. 
   
    Another interesting question would be to study the necessary and sufficient conditions for existence to solutions of \eqref{e:e1} in
    a bounded domain, for instance $\Omega \times [-1,0]$. Since the equation \eqref{e:e1}
   is nonlocal in both space and time, the boundary conditions for a solution in $\Omega \times [-1,0]$ would be defined on 
   $(\Omega^c \times (-1,0))\cup(\Omega \times    (-\infty,-1))$. One could study the boundary conditions for even more general space-time domains.

 \subsection{Notation}
  We here define notation that will be consistent thoughout the paper. 
   \begin{itemize}
    \item $\partial_t^{\alpha}$ - the Caputo derivative as defined in \eqref{e:fracform}.
    \item $\alpha$ - will always denote the order of the fractional time (Caputo) derivative.
    \item $\sigma$ - will always denote the order of the nonlocal spatial operator. 
    \item $t,s$ - will always be variables reserved as time variables. 
    \item $M^{\pm}$ - Pucci's extremal operators as defined in Section \ref{s:viscosity}.
    \item $\lambda, \Lambda$ - Ellipticity constants for the nonlocal operator. 
    \item $Q_r(x_0,t_0)$ - the space-time cylinder $B_r(x_0) \times (t_0 - r^{2\sigma/\alpha} , t_0)$.
    \item $Q_r$ - the cylinder centered at the origin $Q_r(0,0)$. 
   \end{itemize}
   
 \subsection{Outline}
  The outline of our paper is as follows:
  In Section \ref{s:viscosity} we explain the notion of viscosity solution that will be used in the paper.
  In Section \ref{s:caputo} we prove a few results regarding the Caputo derivative that will be necessary throughout the paper. 
  In Section \ref{s:holder} we prove our main result; namely, that solutions are locally H\"older continuous. 
  In Section \ref{s:higher} we prove higher regularity in time for a specific linear equation.

\section{Viscosity Solutions and Pucci's Extremal Operators}  \label{s:viscosity}
 We recall the definition of Pucci's Extremal Operators for fractional nondivergence elliptic operators as introduced in \cite{cs09}. 
 We denote the second order
 difference $\delta(u,x,y):= u(x+y)+u(x-y)-2u(x)$. We assume for the two constants $0 < \lambda < \Lambda$. We define
  \[
   \begin{aligned}
    M^+ u(x) &:= \int_{\R^n} \frac{\Lambda \delta(u,x,y)_+ - \lambda\delta(u,x,y)_-}{|y|^{n+2\sigma}} \\
    M^- u(x) &:= \int_{\R^n} \frac{\lambda \delta(u,x,y)_+ - \Lambda\delta(u,x,y)_-}{|y|^{n+2\sigma}}.
   \end{aligned}
  \]
 These operators give rise to the equations
  \begin{alignat}{2}  
   &\partial_t^{\alpha} u(x,t) - M^+ u(x,t) &\leq f(x,t)   \label{e:p1} \\
   &\partial_t^{\alpha} u(x,t) - M^- u(x,t) &\geq f(x,t)   \label{e:p2} 
  \end{alignat}
 
 We now define a viscosity solution for our problem. 
 We say that a continuous function $u$ is a viscosity subsolution (supersolution) if whenever $\phi$ defined on
 the cylinder $\C:=[t_1,t_2] \times B_r(x_1)$ touches $u$ from above (below)
 at $(x_0,t_o) \in \C$ and we define 
  \[ 
   v(x,t) := 
    \begin{cases}
      \phi(x,t) & \text{if } (x,t) \in \C \\ 
      u(x,t) & \text{otherwise}.\\
    \end{cases}
  \]
 Then  
  \[
   \partial_t^{\alpha} v(x,t) - Iv(x,t) \leq (\geq) f(x,t).
  \]
 A solution is both a subsolution and supersolution. It follows from the definitions that if $u$ is a subsolution 
 (supersolution) to \eqref{e:e1}, then $u$ is a subsolution to \eqref{e:p1} (supersolution to \eqref{e:p2}). It is clear that if a function $u$ is a 
 solution to \eqref{e:e1} in the classical sense, then it is also a solution in the viscosity sense. 
 
 We remark that upper and lower semi-continuity are sufficient for the definitions of sub and super-solution. However, continuity 
 will suffice for the purposes of this paper. 
 
 The following Lemma will be useful in proving higher regularity in time in Section \ref{s:higher}.
 
  \begin{lemma}  \label{l:limit}
   Let $u_k$ be a sequence of continuous bounded viscosity solutions to \eqref{e:e1} in $\R^n \times (-\infty, T)$ converging to $u_0$ bounded and continuous. Then 
   $u_0$ is a viscosity solution to \eqref{e:e1} in $\R^n \times (-\infty, T)$.
  \end{lemma}
  
  \begin{proof}
   The proof is straightforward from the definition of viscosity solutions (see \cite{cc95}). 
  \end{proof}

\section{Caputo Derivative} \label{s:caputo}
 In this section we collect some results regarding the Caputo derivative which we will later use. 
 The following Proposition is an immediate consequence of the formulation of the Caputo derivative given in  \eqref{e:fracform}.
  \begin{proposition}   \label{p:min}
   For a fixed point $t_0$, if $u(t)\geq u(t_0)$ for all $t<t_0$, then 
    \[
     \partial_t^{\alpha} u(t_0) \leq 0.
    \]
   with equality if and only if $u(t)=u(t_0)$ for all $t<t_0$. 
  \end{proposition}
 We will need a solution to a fractional ordinary differential equation
  in Lemma \ref{l:decrease}. This solution will help carry information over various time slices. 
  \begin{equation}  \label{e:ode}
   \begin{cases}
    \partial_t^{\alpha} u + C_1 u = f &\text{  in  } [a,T] \\
    u(a)=0. 
    \end{cases}
  \end{equation}
 An explicit solution of the above equation is given in \cite{d04} as 
  \begin{equation}  \label{e:classic}
     \alpha \int_a^t (t-s)^{\alpha -1} E_{\alpha}'(-C_1(t-s)^{\alpha})f(s) \ ds, 
  \end{equation}
 where  $E_{\alpha}$ are the Mittag-Leffler functions defined by
  \[
   E_{\alpha}(t):= \sum_{j=0}^{\infty} \frac{t^j}{\Gamma(j\alpha +1)}. 
  \]
 As $\alpha \to 1$ this approaches the well know formula
  \[
   u(t)=
   \int_a^t e^{-C_1(t-s)}f(s) \ ds.
  \]
 We will want a lower bound on the solution to \eqref{e:ode} depending on $f$. To begin with
 we obtain the following lower bound. 
  \begin{lemma}
   For $0<\alpha \leq 1$ we have $E_{\alpha}(t)>0$ and $E_{\alpha}'(t)>0$. Furthermore, 
    \begin{equation}  \label{e:lowbound}
     E_{\alpha}(t), E_{\alpha}'(t) \geq C \text{  for  } t \in [-2,2]
    \end{equation}
   with the constant $C>0$ depending on $\alpha$ and uniform as $\alpha \to 1$.
  \end{lemma}
  
  \begin{remark}
   This Lemma need not be true for $\alpha >1$. For instance  $E_{2}(-t^2)= \cos(t)$.
  \end{remark}
  
  \begin{proof}
   From the power series representation of $E$ this is clear for $t \geq 0$. We recall from \cite{d04} that 
    $\partial_t^{\alpha}  E_{\alpha}(-\gamma t^{\alpha})=-\gamma E_{\alpha}(-\gamma t^\alpha)$. We relabel $y(t):= E_{\alpha}(-\gamma t^\alpha)$.
   Then 
    \begin{equation}  \label{e:mittag}
     \partial_t^{\alpha} y(t)= \frac{1}{\Gamma(1-\alpha)} \int_0^t \frac{y'(s)}{(t-s)^{\alpha}} ds = -\gamma y(t).
    \end{equation}
   Now from the power series representation $E_{\alpha}'(0) > 0$, and thus by continuity of the 
   derivative we have $E_{\alpha}'(t) >0$ for $- \delta \leq t$. Therefore, there exists a $\delta$ such that
   $y(t)$ is a nonincreasing function on 
   $(0,\delta)$. Suppose now by way of contradiction that $y(t)$ is not nondecreasing. Then by continuity of $y(t)$, there exists a point $t_2>\delta$ such that
    \[
     w(t) := \max\{y(t_2), y(t)\} 
    \]
   is a nonincreasing function, and there exists $t_1<t_2$ such that $y(t_1)<y(t_2)$.
   From Proposition \ref{p:min} it follows that $\partial_t^{\alpha} w(t) \geq - \gamma w(t)$. We now subtract $\partial_t^{\alpha} y(t)$
   from both sides and evaluate at $t_2$ to obtain
    \[
     \partial_t^{\alpha} [w-y](t_2) \geq -\gamma[w(t_2) -y(t_2)] =0. 
    \]
   However, since $y(t_1)<w(t_1)$ we have from Proposition \ref{p:min} that
    \[
     \partial_t^{\alpha} [w-y](t_2) < 0.
    \] 
   This is a contradiction, and so $y(t)$ is a nonincreasing function for $t\geq 0$. It is then immediate from \eqref{e:fracform} that $y(t)$ is not only 
   nonincreasing, it is also strictly decreasing for $t\geq 0$. 
   Then from the power series representation of $y(t)$ it is immediate that
   $y(t)>0$ and $y'(t)<0$ over $(-\infty, \infty)$. We then obtain that 
    \[ 
     E_{\alpha}(t), E_{\alpha}'(t) \geq C \text{  for  } t \in [-2,2]
    \]
   with the constant $C$ depending on $\alpha$. From the power series representation it follows that this
   constant $C$ is uniform as $\alpha \to 1$.  
  \end{proof}
   
  \begin{corollary}  \label{c:mubound}
   Let $m$ be a solution to $\partial_t^{\alpha}m = -C_1 m +c_0 f(t)$ with $m(-2)=0$,$f \geq 0$ and 
    \[
     \int_{-2}^{-1} f(t) \geq \mu. 
    \]
   Then 
    \begin{equation}  \label{e:expbound}
     m(t) \geq c_0 \mu \frac{\alpha}{2} E_{\alpha}'(-2C_1) \quad \text{for } -1\leq t \leq 0.
    \end{equation}
  \end{corollary}
   
  \begin{proof}
   From the integral representation of $m(t)$ in \cite{d04} and the fact that $E_{\alpha}'(t)>0$, we have 
    \[
     \begin{aligned}
      m(t) &= \alpha \int_{-2}^t (t-s)^{\alpha-1} E_{\alpha}'(-C_1(t-s)^{\alpha})c_0 f(s)ds \\
           &\geq \frac{\alpha}{2}  E_{\alpha}'(-2C_1) \int_{-2}^{t} c_0 f(s) ds \\
           & \geq \frac{\alpha}{2}  E_{\alpha}'(-2C_1) c_0 \mu.
     \end{aligned}
    \]
  \end{proof}

 We will also need the following 
  \begin{proposition}  \label{p:nubound}
   Let $h(t)= \max{|t|^{\nu}-1,0}$. If $t_1\leq 0$ and $\nu <\alpha$ then 
    \[
     0 \geq \partial_t^{\alpha} h(t_1) \geq - c_{\alpha, \nu} 
    \]
   where $c_{\alpha,\nu}$ is a constant depending only on $\alpha$ and $\nu$ but remains uniform as $\alpha \to 1$.
  \end{proposition}
 
  \begin{proof}
   In \cite{acv16} it is shown that 
    \[
     0 \geq \partial_t^{\alpha} h(t) \geq \partial_t^{\alpha} h(-1) \geq c_{\alpha,\nu}.
    \]
   Now 
    \[
     \partial_t^{\alpha} h(-1) 
      = \frac{1}{\Gamma(1-\alpha)} \int_{-\infty}^{-1} \frac{-\nu |s|^{\nu-1}}{(t-s)^{\alpha}} ds \to -\nu \quad \text{as} \quad \alpha \to 1.
    \]
  \end{proof}

\section{H\"older Continuity}  \label{s:holder}
 In this section we follow the method used in \cite{s11} to prove our main result. 
  \begin{lemma}  \label{l:decrease}
   Let $u$ be a continuous function, $u\leq 1$ in $(\R^n \times [-2,0])\cup (B_2 \times [-\infty,0])$, which satisfies 
   the following inequality in the viscosity sense in $B_2 \times [-2,0]$
    \[
     \partial_t^{\alpha} u - M^+ u \leq \epsilon_0.
    \]
   Assume also that  
    \[
     | \{u\leq 0\} \cap (B_1 \times [-2,-1])| \geq \mu.
    \]
   Then if $\epsilon_0$ is small enough there exists $\theta>0$ such that $u \leq 1-\theta$ in $B_1 \times [-1,0]$. 
   The maximum value of $\epsilon_0$ as well as $\theta$ depend on $\alpha, \lambda, \Lambda,n$ and $\sigma$ but remain uniform as $\alpha \to 1$. 
  \end{lemma}

  \begin{proof}
   We consider the fractional ordinary differential equation $m_1: [-2,0] \to \R$
     \[
     \begin{aligned}
      m_1(-2)&=0 \\
      \partial_t^{\alpha} m_1(t) &= c_0 |\{x \in B_1: u(x,t)\leq 0\}| - C_1m(t)
     \end{aligned}
     \]
   From the hypothesis and Corollary \ref{c:mubound} we have
    \[
     m_1(t) \geq c_0 \mu \frac{\alpha}{2} E_{\alpha}'(-2C_1) \quad \text{for } -1\leq t \leq 0.
    \]
   Since the right hand side is not necessarily continuous, in order to ensure that we can evaluate $\partial_t^{\alpha}$ classically at every point 
   we approximate $m_1(t)$ uniformly from below by Lipschitz subsolutions as in \cite{a17}. There exists 
   $m_2:[-2,0] \to \R$ with $m_2\geq0$ and 
   \begin{equation}  \label{e:fode}
     \begin{aligned}
      m_2(-2)&=0 \\
      \partial_t^{\alpha} m_2(t) &\leq c_0 |\{x \in B_1: u(x,t)\leq 0\}| - C_1m_2(t)
     \end{aligned}
    \end{equation}
    and 
     \[
     m_2(t) \geq c_0 \mu \frac{\alpha}{4} E_{\alpha}'(-2C_1) \quad \text{for } -1\leq t \leq 0.
    \]
   We want to show that $u \leq 1- m_2(t) + \epsilon_0 c_{\alpha}2^{\alpha}$ if $c_0$ is small and $C_1$ is large. We can then set 
    $\theta = c_0 \mu \alpha E_{\alpha}'(-2C_1)/8$ for $\epsilon_0$ small to obtain the result of the Lemma. 
   We pick the constant $c_{\alpha}$ such that $\partial_t^{\alpha} c_{\alpha}(2+t)_+^{\alpha} = 1$ for $t>-2$. The constant $c_{\alpha}$ is uniform as $\alpha \to 1$ 
    (see\cite{d04}). 
   Let $\beta:\R \to \R$ be a fixed smooth nonincreasing function such that $\beta(x)=1$ if $x\leq 1$ and $\beta(x)=0$ if $x\geq 2$. 
   Let $b(x)= \beta(|x|)$. Where $b=0$ we have $M^- b>0$. Since $b$ is smooth $M^-b$ is continuous and it remains positive for $b$ small enough (\cite{s11}).
   Thus there exists $\beta_1$ such that $M^- b \geq 0$ if $b \leq \beta_1$.
   
   Assume that there exists some point $(x,t) \in B_1 \times [-1,0]$ such that 
    \[
     u(x,t) > 1 -m_2(t)+\epsilon_0 c_{\alpha} (2+t)_+^{\alpha}. 
    \]
   We will arrive at a contradiction
   by looking at the maximum of the function
    \[
     w(x,t) =u(x,t) +m_2(t)b(x) - \epsilon_0 c_{\alpha}(2+t)_+^{\alpha}
    \]
   We assume there exists a point in $B_1 \times [-1,0]$ where $w(x,t)>1$. Let $(x_0,t_0)$ be the point that realizes the maximum of $w$:
    \[
     w(x_0,t_0) = \max_{\R^n \times (-\infty,0]}w(x,t).
    \]  
   This maximum is larger than 1, and so it must be achieved when $t> -2$ and $|x|<2$. 
   
   Let $\phi(x,t):= w(x_0,t_0)-m_2(t)b(x)+c_{\alpha}(2+t)_+^{\alpha}$. As before we define $\phi(x,t)=\phi(x,-2)$ for 
   $t\leq -2$. The function $\phi$ touches $u$ from above at the point $(x_0,t_0)$. 
   We define
    \[
     v(x,t) := 
              \begin{cases}
                 \phi(x,t) & \text{if } x \in B_r  \\ 
                    u(x,t) & \text{if } x \notin B_r.\\
    \end{cases}
    \] 
   Then from the definition of viscosity solution we have
    \begin{equation}  \label{e:viscos}
     \partial_t^{\alpha}v - M^+v \leq \epsilon_0 \quad \text{ at } \quad (x_0,t_0). 
    \end{equation}
   We also have that 
    \begin{equation}  \label{e:capu}
     \partial_t^{\alpha} v(x_0,t_0) = -\partial_t^{\alpha} m_2(t_0)b(x_0) +\epsilon_0
    \end{equation}
   Now exactly as in \cite{s11} we obtain the following bound for $G := \{x \in B_1 \mid u(x,t_0)\leq 0 \}$,
    \begin{equation}  \label{e:silv}
     M^+ v(x_0,t_0) \leq -m_2(t_0) M^-b(x_0,t_0) -c_0|G \setminus B_r|
    \end{equation}
   for some universal constant $c_0$. This is how we choose $c_0$ in the fractional ordinary differntial equation. 
   We now look at two different cases and obtain a contradiction in both. Suppose $b(x_0,t_0)\leq\beta_1$. Then $M^-b(x_0,t_0)\geq0$, and so
   from \eqref{e:silv} we have
    \[
     M^+v(x_0,t_0)\leq -c_0|G\setminus B_r|. 
    \]
   Combining this with \eqref{e:viscos},\eqref{e:capu}, and \eqref{e:fode} we obtain
    \[
     0 \geq \left(-c_0 |\{x \in B_1: u(x,t)\leq 0\}| + C_1m_2(t)\right)b(x_0)  + c_0|G\setminus B_r|.
    \]
   For any $C_1>0$ this will be a contradiction by taking $r$ small enough. 
   
   Now suppose $b(x_0)>\beta_1$. Since $b$ is a smooth compactly supported function, there exists $C$ such that $|M^-b|\leq C$. We then have 
   the bound from \eqref{e:silv} that
    \[
     M^+ v(x_0,t_0) \leq Cm_2(t_0) +c_0|G \setminus B_r|
    \]  
   and inserting this in \eqref{e:viscos} with \eqref{e:capu} and \eqref{e:fode}we have
    \[
     0 \geq \left(-c_0 |\{x \in B_1: u(x,t)\leq 0\}| + C_1m_2(t)\right)b(x_0)   - Cm_2(t_0)+c_0|G \setminus B_r|
    \]
   Letting $r \to 0$ we obtain 
    \[
     \begin{aligned}
     0  &\geq c_0(1-b(x_0))|G| + (C_1b(x_0)-C)m_2(t_0) \\
        &\geq c_0(1-b(x_0))|G| + (C_1\beta_1-C)m_2(t_0).
     \end{aligned}
    \]
   Choosing $C_1$ large enough we obtain a contradiction. 
  \end{proof}

 We now define
  \[
   Q_r := B_r \times [-r^{2\sigma/\alpha},0]
  \]

 \begin{lemma}  \label{l:down}
  Let $u$ be bounded and continuous on $Q_1$ and satisfying 
   \begin{equation}
    \begin{aligned}
     &\partial_t^{\alpha} u - M^+ u \leq \epsilon_0/2, \\
     &\partial_t^{\alpha} u - M^+ u \geq -\epsilon_0/2.
    \end{aligned}
   \end{equation}
  in the viscosity sense in $Q_1$. 
  Then there are univeral constants $\theta>0$ and $\nu>0$ depending only on $n,\sigma,\Lambda,\lambda,\alpha$, but uniform as $\alpha \to 1$
  such that if 
   \[
    \begin{aligned}
     |u| \leq 1                     &\quad \text{in} \quad B_1 \times [-1,0] \\
     |u(x,t)| \leq 2|4x|^{\nu}-1 &\quad \text{in} \quad (\R^n\setminus B_1) \times [-1,0] \\
     |u(x,t)| \leq 2|4t|^{\nu}-1 &\quad \text{in} \quad B_1 \times (-\infty,-1] 
    \end{aligned}   
   \]
  then 
   \[
    \text{osc}_{Q_{1/4}} u \leq 1-\theta
   \]
 \end{lemma}

 \begin{proof}
  We consider the rescaled version 
   \[
    \tilde{u}(x,t) := u(4x,4^{2\sigma/\alpha} t). 
   \]
  The function $\tilde{u}$ will stay either positive or negative in half of the points in $B_1 \times [-2,-1]$. Let us assume that 
   $\{\tilde{u} \leq 0\}\cap (B_1 \times [-2,-1])\geq |B_1|/2$. Otherwise we can repeat the proof for $-\tilde{u}$. We would like to apply Lemma \ref{l:decrease}.
   To do so we would need $\tilde{u}\leq 1$. We consider $v:= \min\{1,\tilde{u}\}$. Inside $Q_{4}$ we have $v=\tilde{u}$. The error comes only from the tails in 
   the computations. Exactly as in \cite{s11} we obtain for $\kappa$ small enough
    \[
     -M^+ v \leq -M^+ \tilde{u} + \epsilon_0/4.
    \]
   From the Proposition \ref{p:nubound} we have for small enough $\kappa$
    \[
     \partial_t^{\alpha} v \leq \partial_t^{\alpha} \tilde{u} + \epsilon_0/4.
    \]
   Thus 
    \[
     \partial_t^{\alpha} v - M^+ v \leq \epsilon_0.
    \]
   We now apply Lemma \ref{l:decrease} to $v$ to conlude the proof. 
 \end{proof}

 \begin{theorem}   \label{t:holder}
  Let $u$ be bounded and continuous on $B_2 \times [-2,0]$ as well as a solution to \eqref{e:e1} in $B_2 \times [-2,0]$. Assume also that  
   \[
    \begin{aligned}
     |u| \leq M                     &\quad \text{in} \quad B_2 \times [-2,0] \\
     |u(x,t)| \leq M|4x|^{\nu}-1 &\quad \text{in} \quad (\R^n\setminus B_2) \times [-2,0] \\
     |u(x,t)| \leq M|4t|^{\nu}-1 &\quad \text{in} \quad B_2 \times (-\infty,-2]. 
    \end{aligned}   
   \]
  Then there are two constants  $C,\kappa>0$ depending only on $\Lambda,\lambda, n,\alpha,\sigma$, but uniform as $\alpha \to 1$
  such that for every $t>0$, $u$ is H\"older continuous and we have the estimate for 
   \[
    |u(x,t)-u(y,s)| \leq C (M + \epsilon_0^{-1} \| f\|_{L^{\infty}})\frac{|x-y|^{\kappa}+|t-s|^{\kappa \alpha/(2\sigma)}}{t^{\kappa}}.
   \] 
 \end{theorem}

 \begin{proof}
  We first choose $\kappa<\nu$ for $\nu$ as in Lemma \ref{l:down}.
  We consider the rescaled function 
   \[
    v(x,t) = \frac{u(x_0 + t_0^{-2\sigma/\alpha}x,t_0(t+1))}{M + \epsilon_0^{-1} \| f\|_{L^{\infty}}}
   \]
  Notice that $v$ is a solution to 
   \[
    \partial_t^{\alpha} v - Iv = f(x,t)\frac{t_0^\alpha}{M + \epsilon_0^{-1} \| f\|_{L^{\infty}}}
   \]
  in $B_2 \times [-1,0]$. We let $r=1/4$ and the estimate will follow as soon as we show
   \begin{equation}  \label{e:osc2}
    \text{osc}_{Q_{r_k}} v \leq 2 r^{\kappa k}.
   \end{equation}
  \eqref{e:osc2} will be proven by constructing two sequences $a_k \leq v \leq b_k$ in $Q_{r_k}$, $b_k - a_k= 2r^{\kappa k}$ with 
  $a_k$ nondecreasing and $b_k$ nonincreasing. The sequence is constructed inductively. 
  
  Since $|v|\leq 1$ in $B_2 \times [-2,0]$, we can start by choosing some $a_0 \leq \inf v$ and $b_0 \geq \sup v$ so that $b_0 - a_0=2$. Assuming now that
  the sequences have been constructed up to the value $k$  we scale
   \[
    w(x,t) = (v(r^k x,r^{2k\sigma /\alpha}t) -(a_k+b_k)/2)r^{-\kappa k}.
   \] 
  We then have 
   \[
    \begin{aligned}
     |w| \leq 1                 &\quad \text{in } \quad  Q_1 \\
     |w| \leq 2r^{-\kappa k} -1 &\quad \text{in } \quad  Q_{r^{-k}}.
    \end{aligned}
   \]
  and so
   \[
    \begin{aligned}
     |w(x,t)|\leq 2|x|^{\nu}-1  &\quad \text{for } \quad (x,t) \in B_1^c \times [-1,0] \\
     |w(x,t)|\leq 2|t|^{\nu}-1  &\quad \text{for } \quad (x,t) \in B_1 \times (-\infty,-1).
    \end{aligned}
   \]
  Notice also that $w$ has new right hand side bounded by 
   \[
    \epsilon_0 r^{k(\kappa-2\sigma)}
   \]
  which is strictly smaller than $\epsilon_0$ for $\kappa<2\sigma$. 
   For $\kappa$ small enough we can apply Lemma \ref{l:down} to obtain 
   \[
    \text{osc}_{ Q_r}  w \leq 1-\theta. 
   \]
  Then if $\kappa$ is chosen smaller than the $\kappa$
  in Lemma \ref{l:down} and also so that $1-\theta \leq r^{\kappa}$, then this implies
   \[
    \text{osc}_{ Q_{r^{k+1}}} w \leq r^{\kappa(k+1)}
   \]
  so we can find $a_{k+1}$ and $b_{k+1}$ and this finishes the proof.  
 \end{proof}

\section{Higher Regularity in Time}  \label{s:higher}
 As this paper is concerned with the Caputo derivative we will focus on higher regularity in time rather than higher regularity in spatial variables which has 
 already been treated in the literature. 
 
 In this section we fix a smooth cut-off function $\eta$ such that $\eta\equiv 1 $ for $1/2\leq t \leq 1$ and $\eta \equiv 0$ for $t \leq 1/4$. 
  \begin{lemma}  \label{l:byparts}
   Let $0<\gamma\leq 1$ and let $g:(-\infty,1)\to \R$ be such that 
    \[
      \| g\|_{C^{0,\gamma}[1/8,1]} \leq C \quad \text{  and  } \quad |g(t)| \leq C(2+|t|^{\nu})
    \]
   with $\nu <\alpha$. 
   Then 
    \[
     \partial_t^{\alpha} (\eta g) = \eta \partial_t^{\alpha} g + \tilde{g} 
    \]
   where 
    \[
      \| \tilde{g} \|_{C^{0,\gamma}(-\infty,1)} \leq C_{\alpha}
    \] 
   with $C_{\alpha}$ dependent on $\alpha$ but uniform as $\alpha \to 1$. 
  \end{lemma}

  \begin{proof}
   We explicitly compute 
    \[ 
     \begin{aligned}
      \partial_t^{\alpha} (\eta g) &= \frac{\alpha}{\Gamma(1-\alpha)}\int_{-\infty}^t \frac{(\eta g)(t)-(\eta g)(s)}{(t-s)^{1+\alpha}} ds \\
             &= \eta(t) \frac{\alpha}{\Gamma(1-\alpha)}\int_{-\infty}^t \frac{ g(t)- g(s)}{(t-s)^{1+\alpha}} ds
               + \frac{\alpha}{\Gamma(1-\alpha)}\int_{-\infty}^t \frac{g(s)(\eta(t)-\eta(s))}{(t-s)^{1+\alpha}} ds \\
             &= \eta(t)\partial_t^{\alpha}g(t) +   \frac{\alpha}{\Gamma(1-\alpha)}\int_{-\infty}^t \frac{g(s)(\eta(t)-\eta(s))}{(t-s)^{1+\alpha}} ds \\
     \end{aligned}
    \]
   We now write $g=g_1+g_2$ where $g_1:= g\chi_{\{t<1/8\}}$. We have 
    \[
     \frac{\alpha}{\Gamma(1-\alpha)}\int_{-\infty}^t \frac{g_1(s)(\eta(t)-\eta(s))}{(t-s)^{1+\alpha}} ds
      =\frac{\alpha}{\Gamma(1-\alpha)}\int_{-\infty}^{1/8} \frac{g_1(s)\eta(t)}{(t-s)^{1+\alpha}} ds
    \]
   This is clearly a differntiable function for $t\geq 1/4$ (recall that $\eta(t)\equiv 0$ for $t\leq 1/4$). The H\"older norm
   will depend only on the growth of $|g_1(t)|\leq C(2+|t|^\nu)$. Furthermore, this term clearly goes to zero independent of $g$ as $\alpha \to 1$. 
   
   We now define
    \[
     \tilde{\eta}(t,s):=\frac{\eta(t)-\eta(s)}{(t-s)},
    \] 
   and note that $\tilde{\eta}(t,s)$ is a smooth function. Then
    \[
     \begin{aligned}
      &\frac{1}{|h|^\gamma} \frac{\alpha}{\Gamma(1-\alpha)} 
       \left[\int_{-\infty}^{t+h} \frac{g_2(s)(\eta(t+h)-\eta(s))}{(t+h-s)^{1+\alpha}} ds 
        - \int_{-\infty}^{t} \frac{g_2(s)(\eta(t)-\eta(s))}{(t-s)^{1+\alpha}} ds\right] \\
      &=   \frac{1}{|h|^\gamma} \frac{\alpha}{\Gamma(1-\alpha)} 
       \left[\int_{-\infty}^{t+h} \frac{g_2(s)\tilde{\eta}(t+h,s)}{(t+h-s)^{\alpha}} \ ds 
        - \int_{-\infty}^{t} \frac{g_2(s)\tilde{\eta}(t,s)}{(t-s)^{\alpha}} \ ds\right] \\
      &=\frac{1}{|h|^\gamma} \frac{\alpha}{\Gamma(1-\alpha)} 
       \left[
       \int_{1/8-h}^{t} \frac{g_2(s+h)\tilde{\eta}(t+h,s+h)-g_2(s)\tilde{\eta}(t,s)}{(t-s)^{\alpha}} ds 
        \right] 
     \end{aligned}
    \]
   The last term is bounded by the H\"older norm of $g_2 \tilde{\eta}$ 
   and is uniformly bounded as $\alpha \to 1$. 
  \end{proof}

  \begin{lemma}
   Let $u$ be a solution to \eqref{e:e1} in $[-1,0] \times U$ for some open $U \subset \R^n$. 
   Then $\eta u$ is a solution to \eqref{e:e1} in $(-\infty,0]\times U$ with right hand side 
    \[
     \eta(t)f(t,x) + \frac{\alpha}{\Gamma(1-\alpha)}\int_{-\infty}^t \frac{u(x,s)[\eta(t)-\eta(s)]}{(t-s)^{1+\alpha}}
    \]
  \end{lemma}

  \begin{proof}
    If $t_0 \notin $ supp $\eta$ this is immediately clear. Otherwise $\eta \psi$ touches $\eta u$ from above if and only if $\psi$ touches $u$ from above. Suppose
    $\eta \psi$ touches $\eta u$ from above at a point $(x,t)$. As explained in \cite{a17}, if $\psi$ touches $u$ from above at $(x,t)$, then $\partial_t^{\alpha} u$
    can be evaluated classically at $(x,t)$. Furthermore, it follows that $\partial_t^{\alpha} u(x,t) - I \psi(x,t)\leq f(x,t)$. Then
    \[
     \begin{aligned}
      &\partial_t^{\alpha} \eta u(x,t) - I \eta \psi(x,t)  \\
         &\qquad= \eta(t) [\partial_t^{\alpha} u(x,t)- I\psi(x,t)]
           + \frac{\alpha}{\Gamma(1-\alpha)}\int_{-\infty}^t \frac{u(x,s)[\eta(t)-\eta(s)]}{(t-s)^{1+\alpha}} \ ds. 
     \end{aligned}
    \]
    It then follows that 
    \[
     \partial_t^{\alpha} \eta\psi(x,t) - I \eta \psi(x,t) \leq \eta(t) f(x,t) + \frac{\alpha}{\Gamma(1-\alpha)}\int_{-\infty}^t \frac{u(x,s)[\eta(t)-\eta(s)]}{(t-s)^{1+\alpha}} \ ds. 
    \]
     When $\eta \psi$ touches from below the proof is similar.
  \end{proof}

  \begin{theorem}   \label{t:higher}
   Let $u$ be a solution to  
    \begin{equation}  \label{e:linear}
     \partial_t^{\alpha} u -  \int_{\R^n}\frac{u(x+y)-u(x)}{|y|^{n+2\sigma}} a^{ij}(x,y) = f
    \end{equation}
     in $\R^n \times [-1,T]$ with right hand side $f \in C^{k,\beta}$. 
   Assume 
    \[
     u(x,t) \leq 2(|t|-1)^{\nu} \quad \text{ for } \nu <\alpha \quad \text{ and } \quad t<0.
    \]
   Let $-1<t_0$. We label $\gamma = 2\sigma \kappa/\alpha$, and recall that $0<\gamma<1$. Let $\epsilon>0$. Then there exists
   a constant $C$ depending on $t_0,k,n, \lambda, \Lambda, \sigma,\nu, \epsilon, \alpha, $ but uniform as $\alpha \to 1$ such that 
    \[
     \| u \|_{C^{k,\beta+\kappa}(\R^n \times [t_0,T-\epsilon])} \leq C 
    \]
  \end{theorem}
  
   \begin{remark}
    We consider solutions in all of $\R^n \times [-1,T]$. This will ensure sufficiently regular ``boundary conditions''. 
    The method in the following proof would apply to solutions in bounded domains provided that the ``boundary'' data is sufficiently regular. Otherwise, solutions
    need not have higher regularity (see \cite{cl142}).
   \end{remark}
   
  \begin{proof}
   As before we rescale so that $t_0$ is distance one from $-1$. This is where the dependence on $t_0$ comes. 
   We multiply by $\eta$ and our new solution $w = \eta v$ has a new right hand side $g \in C^{0,\kappa}$. We label $\gamma = \alpha/(2\sigma)$, and note that
   for every fixed $x$
    \[
     u(t,x) \in C^{0,\gamma}(-\infty,T) 
    \] 
    We now use the ideas of a difference quotient shown
   in \cite{cc95}. For $h>0$  
    \[
     \frac{w(x,t+h)-w(x,t)}{h^{\gamma}}
    \] 
   is a bounded solution in $\R^n\times (-\infty, T-h)$. Then from the proof in \cite{cc95}, for every fixed $x$,
    \[
     \| w \|_{C^{0,2\gamma}(\{x\} \times (-\infty, T-\epsilon_2))} \leq C
    \]
   with $C$ depending only on the $C^{0,\gamma}$ norm of $w$ and $\epsilon_2$. Then $w$ satisfies \eqref{e:linear} 
   with right hand side in $C^{0,2\gamma}$. This can be 
   repeatedly finitely many $m$ times to prove
    \[
     \| w \|_{C^{0,m\gamma}(\R^n \times (-\infty, T-\epsilon_m))} \leq C_m
    \]
   as long as $f \in C^{0,m\gamma}$. If $f$ is Lipschitz then we can repeat the process finitely many times to obtain $w \in C^{0,1}$. 
   Then taking the difference quotient one more time
    \[
     w_h := \frac{w(x,t+h)-w(x,t)}{h}
    \]
   is a solution to \eqref{e:linear} with bounded right hand side. Hence by Theorem \ref{t:holder}, $w_h$ is uniformly H\"older continuous and
   so a limit solution $u_0 \in C^{0,\gamma}$ exists from Lemma \ref{l:limit}. Hence, $h \to 0$  we conclude that for every fixed $x$
    \[
     w(x,t) \in C^{1,\gamma}((-\infty, T- \epsilon_k)).
    \]
   We can then repeat the process as long as the original right hand side $f$ is regular enough.   
  \end{proof}

 We remark that as in \cite{cc95} and \cite{s11} once one proves a uniqueness theorem of Jensen type, one could prove $C^{1,\gamma}$ regularity in
 time for solutions of the nonlinear equation \eqref{e:e1} provided the kernels $K$ were time-independent.

\bibliographystyle{amsplain}
\bibliography{refnondiv}

\end{document}